\documentclass{article}
\usepackage{amssymb}
\usepackage{amsmath}

\newcommand{\Eff}{{\cal E}\! f\! f}

\newtheorem{proposition}{Proposition}[section]
\newtheorem{lemma}[proposition]{Lemma}
\newtheorem{corollary}[proposition]{Corollary}
\newtheorem{definition}[proposition]{Definition}
\newtheorem{theorem}[proposition]{Theorem}
\newtheorem{exrcise}{Exercise}

\newcounter{opgaveteller}
\newcounter{lijst-teller}

\newcounter{boon}
\newcounter{boon2}
\newenvironment{proof}{\noindent {\bf Proof}. \nopagebreak }{\nopagebreak\hfill\rule{2mm}{3mm}}

\input xypic
\xyoption{all}
\UseComputerModernTips
\CompileMatrices
\definemorphism{dat}\dashed\tip\notip
\definemorphism{dub}\Solid\Tip\notip

\newenvironment{rlist}%
   {\begin{list}{\roman{lijst-teller}\/)\hfil}%
              {\labelwidth 2em%
               \leftmargin\labelwidth\advance\leftmargin by\labelsep%
               \usecounter{lijst-teller}}}%
   {\end{list}}

\newenvironment{r'list}%
   {\begin{list}{\roman{lijst-teller}\/)$'$\hfil}%
              {\labelwidth 2em%
               \leftmargin\labelwidth\advance\leftmargin by\labelsep%
               \usecounter{lijst-teller}}}%
   {\end{list}}

\newenvironment{alist}
   {\begin{list}{\alph{boon}\/)\hfil}%
               {\labelwidth 2em%
               \leftmargin\labelwidth\advance\leftmargin by\labelsep%
               \usecounter{boon}}}%
   {\end{list}}

    {\end{list}}

\title{Basic Subtoposes of the Effective Topos}
\author{Sori Lee \and Jaap van Oosten\footnote{corresponding author: Department of Mathematics, Utrecht University, P.O.\ Box 80.010, 3508 TA Utrecht, The Netherlands, {\tt j.vanoosten@uu.nl}}}

\begin{document}
\newcommand{\isi}{\mbox{${\rm I}\Sigma _1^i$}}
\maketitle
\section*{Introduction}
A fundamental concept in Topos Theory is the notion of {\em subtopos}: a subtopos of a topos $\cal E$ is a full subcategory which is closed under finite limits in $\cal E$, and such that the inclusion functor has a left adjoint which preserves finite limits. It then follows that this subcategory is itself a topos, and its internal logic has a convenient description in terms of the internal logic of $\cal E$. Subtoposes of $\cal E$ are in 1-1 correspondence with {\em local operators\/} in $\cal E$: these are certain endomaps on the subobject classifier of $\cal E$.

Whereas local operators/subtoposes of Grothendieck toposes can be neatly described in terms of Grothendieck topologies, for realizability toposes the study of local operators is not so easy. Yet it is important, since many variations on realizability, such as modified realizability, extensional realizability and Lifschitz realizability arise as the internal logic of subtoposes of standard realizability toposes.

Already in his seminal paper \cite{HylandJ:efft} where he introduces the effective topos $\Eff$ (the mother of all realizability toposes), Martin Hyland studied local operators and established that there is an order-preserving embedding of the Turing degrees in the lattice of local operators. Andy Pitts in his thesis (\cite{PittsAM:thet}) has also some material (and in particular an example of a local operator which differs from the examples in Hyland's paper, and which will be studied a bit further in the present paper); there is a small note by Wesley Phoa (\cite{PhoaW:relcet}); and finally, the second author of the present paper identified the local operator which corresponds to Lifschitz' realizability (\cite{OostenJ:extlrh,OostenJ:remlrt}). But as far as we are aware, this is all.

The lattice of local operators in $\Eff$ is vast and notoriously difficult to study. We seem to lack methods to construct local operators and tell them apart. The present paper aims to improve on this situation in the following way: it is shown (theorem \ref{Mfreecom}) that every local operator is the internal join of a family (indexed by a nonempty set of natural numbers) of local operators induced by a nonempty family of subsets of $\mathbb{N}$ (which we call {\em basic\/} local operators). Then, we introduce a technical tool ({\em sights}) by which we can study inequalities between basic local operators. We construct an infinity of new basic local operators and we have some results about what new functions from natural numbers to natural numbers arise in the corresponding subtoposes. For many of our finitary examples (finite collections of finite sets) we can show that they do not create any new number-theoretic functions; for Pitts' example we can show that it forces all {\em arithmetical\/} functions to be total. This seems interesting: we have a realizability-like topos which, though far from being Boolean, yet satisfies true arithmetic (theorem~\ref{pitts=true}). There might be genuine models of nonstandard arithmetic in this topos (by McCarty's \cite{McCartyD:varth}, such cannot exist in $\Eff$: see also \cite{BergB:aricat}). Since Pitts' local operator is induced by the collection of cofinite subsets of $\mathbb{N}$, this is reminiscent of Moerdijk and Palmgren's work on intuitionistic nonstandard models (\cite{MoerdijkI:minsta,MoerdijkI:minmha}) obtained by filters.

There are other reasons why one should be interested in the lattice of local operators in $\Eff$. It is a Heyting algebra in which, as we saw, the Turing degrees embed. It shares this feature with the (dual of the) Medvedev lattice (\cite{MedvedevY:degdmp}), which enjoys a lot of attention these days. Apart from the work by Sorbi and Terwijn (see, e.g., \cite{SorbiA:remsml,TerwijnS:medlcc,TerwijnS:conlml}) who study the logical properties of this lattice, there is the program {\sl Degree Theory: a New Beginning\/} of Steve Simpson, who argues that degree theory should be studied within the Medvedev lattice. From his plenary address `Mass Problems' at the Logic Colloquium meeting in Bern, 2008 (\cite{SimpsonS:masp}): {\sl ``In the 1980s and 1990s, degree theory fell into disrepute. In my opinion, this decline was due to an excessive concentration on methodological aspects, to the exclusion of foundationally significant aspects}. Indeed, it is commonplace in mathematics, in order to study certain structures, to embed them into larger ones with better properties (the passage from ring elements to ideals in number theory; the passage from elements of a structure to types in model theory). By the way, the relationship between the Medvedev lattice and the lattice of local operators in $\Eff$ seems a worthwhile research project.

This paper is organized as follows. Section 1 reminds the reader of some generalities about the subobject classifier $\Omega$, its set of monotone endomaps and local operators, for as much as is relevant to this paper. Section 2 studies these things in the effective topos. Section 3 recalls known facts from the (limited) literature on the subject. In section 4 we introduce our main innovation: the concept of sights. Section 5, Calculations, then presents our results. Finally, we present a concrete definition of truth for first-order arithmetic in subtoposes corresponding to local operators, using the language of sights.
\medskip

\noindent A remark on authorship of the results: most of the technical material was presented in the first author's doctoral thesis (\cite{LeeS:subet}). 

\subsection*{Notation}
In this paper, juxtaposition of two terms for numbers: $nm$ will almost always stand for: the result of the $n$-th partial recursive function to $m$. The only exception is in the conditions in statements in section 5, where `$2m$' really means 2 times $m$, and in the proof of \ref{ceiling} where $dm$ also means $d$ times $m$. We hope the reader can put up with this.

We use the Kleene symbol $\simeq$ between two possibly undefined terms. We use $\langle ,\rangle$ for coded sequences and $(-)_i$ for the $i$-th element of a coded sequence. The symbol $\ast$ between coded sequences means: take the code of the concatenated sequence; so if $a=\langle a_0,\ldots ,a_{n-1}\rangle$ and $b=\langle b_0,\ldots ,b_{m-1}\rangle$ then $a\ast b=\langle a_0,\ldots ,a_{n-1},b_0,\ldots ,b_{m-1}\rangle$. We use $\lambda x.t$ for a standard index of a partial recursive function sending $x$ to $t$.

We employ the logical symbols $\wedge$, $\to$ etc.\ between formulas, but in the context of $\Eff$ also between subsets of $\mathbb{N}$, where
$$\begin{array}{rcl}A\wedge B & = & \{\langle a,b\rangle\, |\, a\in A,b\in B\} \\
A\to B & = & \{ e\, |\,\text{for all $a\in A$, $ea$ is defined and in $B$}\} \end{array}$$

For further, unexplained, standard notations regarding the effective topos, we refer to the treatment \cite{OostenJ:reaics}.

\section{Subobject classifier, monotone maps amd local operators}
We shall use the internal language of toposes freely; we refer to one of several available text books on Topos Theory (\cite{LambekJ:inthoc,MacLaneS:shegl,JohnstonePT:skee}) for expositions of this topic.

If $1\stackrel{\rm true}{\to}\Omega$ is a subobject classifier, elements of $\Omega$ will act as propositions ($\Omega$ is the power set of a one-element set $\{\ast\}$; and $p\in\Omega$ will also denote the proposition ``$\ast\in p$''); hence $\Omega$ is a model of second-order intuitionistic propositional logic. When we use an expression from this logic and say that it `holds', or is `true', we have this standard interpretation in mind.

Top and bottom elements of $\Omega$ are denoted by $\top$ and $\bot$, respectively.
\begin{definition}\label{locdef}\em A {\em local operator} is a map $j:\Omega\to\Omega$ such that the following statements are true:\begin{alist}\item $\forall p.p\to j(p)$\item $\forall pq.j(p\wedge q)\leftrightarrow j(p)\wedge j(q)$\item $\forall p.j(j(p))\to p$\end{alist}
Equivalently, $j$ is a local operator iff the following statements are true:\begin{rlist}\item $\forall pq.(p\to q)\to (j(p)\to j(q))$\item $\top\to j(\top )$\item $\forall p.j(j(p))\to j(p)$\end{rlist}
A {\em monotone map\/} is a map $j:\Omega\to\Omega$ for which i) holds.\end{definition}
We have a subobject Mon of the exponential $\Omega ^{\Omega}$, consisting of the monotone maps, and a subobject Loc of Mon, consisting of the local operators.

We note that Mon is the free suplattice (for suplattices and locales, see \cite{JoyalA:extgtg}) on a poset: the object $\Omega ^{\Omega}$ represents both the endomaps on $\Omega$ and the subobjects of $\Omega$; under this correspondence the  monotone functions are the upwards closed subobjects of $\Omega$. It follows that Mon is the free suplattice on $\Omega ^{\rm op}$ (recall that the free suplattice on a poset $P$ is the set of downwards closed subsets of $P$). In particular, Mon is an internal locale.

We also observe that since $\Omega$ is (internally) complete, Mon is a retract of $\Omega ^{\Omega}$: the retraction sends $g\in\Omega ^{\Omega}$ to the map $p\mapsto\exists q.(g(q)\wedge (q\leq p))$.

Also Loc is an internal locale, as we conclude from the following folklore result in Topos Theory:
\begin{proposition}\label{L1} The inclusion ${\rm Loc}\to {\rm Mon}$ has a left adjoint $L$ which preserves finite meets.\end{proposition} 
\begin{proof} Define $L(f)$ by the second-order propositional expression:
$$L(f)(p)\; =\;\forall q.[((p\to q)\wedge (f(q)\to q))\to q ]$$
It is easy to deduce that $p\to r$ implies $L(f)(p)\to L(f)(r)$, so i) of definition~\ref{locdef} is satisfied; also ii) holds since $L(f)(\top )$ is valid.

For iii), we first prove the implication
$$f(L(f)(p))\to L(f)(p)$$
as follows: assume $f(L(f)(p))$, $f(r)\to r$, $p\to r$. Since $L(f)(p)$ implies $[((p\to r)\wedge (f(r)\to r))\to r]$ and $f$ is assumed to be in Mon, we have $f(r)$, and hence $r$ by assumption. We conclude that $f(L(f)(p))$ implies
$$\forall r.[((p\to r)\wedge (f(r)\to r))\to r]$$
which is $L(f)(p)$, as desired.
Since we know $f(L(f)(p))\to L(f)(p)$ we can instantiate $L(f)(p)$ for $q$ in
$$\forall q.[((L(f)(p)\to q)\wedge (f(q)\to q))\to q]$$
which is the formula for $L(f)(L(f)(p))$, and get $L(f)(L(f)(p))\to L(f)(p)$, as desired. We conclude that $L(f)\in {\rm Loc}$.

For $j\in {\rm Loc}$ and $f\in {\rm Mon}$, the equivalence
$$f\leq j\;\Leftrightarrow\; L(f)\leq j$$
is easy, which establishes the adjunction.

It remains to be seen that $L$ preserves finite meets. It is straightforward that $L$ preserves the top element. For binary meets, consider that these are given pointwise in Mon. So assume $L(f)(p)\wedge L(g)(p)$; we must prove
$$\forall s.[((f(s)\wedge g(s)\to s)\wedge (p\to s))\to s]$$
Assuming $f(s)\wedge g(s)\to s$, or equivalently $f(s)\to (g(s)\to s)$, as well as $p\to s$, $L(g)(p)$ gives $f(s)\to s$. Again using $p\to s$ and $L(f)$ we get $s$, as desired.\end{proof}
\section{Monotone maps, local operators and basic local operators in $\Eff$}
In $\Eff$, the object Mon of monotone maps $\Omega\to\Omega$ is covered by the assembly
$M\; =\; (M,E)$
where $$M\; =\;\{ f:{\cal P}(\mathbb{N})\to {\cal P}(\mathbb{N})\, |\,\bigcap_{p,q\subseteq\mathbb{N}}(p\to q)\to (f(p)\to f(q))\neq\emptyset\}$$
and
$$E(f)\; =\; \bigcap_{p,q\subseteq\mathbb{N}}(p\to q)\to (f(p)\to f(q))$$
Mon is endowed with a preorder structure: we put
$$[f\leq g]\; =\; E(f)\wedge E(g)\wedge\bigcap_{p\subseteq\mathbb{N}}f(p)\to g(p)$$
The actual object Mon of monotone maps is a quotient of $M$ by the equivalence relation $\cong$ induced by this preorder. However, we shall find it more convenient to work with the preorder $M$ than with its quotient.

Actually, since Mon is a retract of $\Omega ^{\Omega}$ which is a uniform object (all power objects in $\Eff$ are, see \cite{OostenJ:reaics}, 3.2.6), instead of $M$ we could have taken a sheaf. In fact, for $f\in M$, $a\in E(f)$ and $$F(f)(p)\equiv \bigcup_{q\subseteq\mathbb{N}}((q\to p)\wedge f(q))$$
we have: if $\beta$ is such that $\beta z\langle x,y\rangle w\simeq \langle z(xw),y\rangle$ then $\beta\in E(F(f))$, and from $a$ we easily find an element of $[f\cong F(f)]$.

Similarly, we have an internal preorder Lo, a sub-assembly of $M$ which covers the object Loc of local operators:
$${\rm Lo}\; =\; (\{ f:{\cal P}(\mathbb{N})\to {\cal P}(\mathbb{N}\, |\, E_1(f)\wedge E_2(f)\wedge E_3(f)\neq\emptyset\} ,E)$$
where
$$\begin{array}{rcl} E_1(f) & = & \bigcap_{p.q\subseteq\mathbb{N}}[(p\to q)\to (f(p)\to f(q))] \\
E_2(f) & = & f(\mathbb{N}) \\
E_3(f) & = & \bigcap_{p\subseteq\mathbb{N}}[f(f(p))\to f(p)] \\
E(f) & = & E_1(f)\wedge E_2(f)\wedge E_3(f) \end{array}$$
and Lo inherits the preorder from $M$.

The reflection map $L:{\rm Mon}\to {\rm Loc}$ lifts to a map $L:M\to {\rm Lo}$, given by
$$L(f)(p)\; =\; \bigcap_{q\subseteq\mathbb{N}}((p\to q)\wedge (f(q)\to q))\to q$$
Then Lo as internal preorder is equivalent to the preorder $(M,\leq _L)$ where $f\leq _Lg$ iff $f\leq L(g)$.
\medskip

\noindent The following form of the map $L$ is essentially due to A.\ Pitts (\cite{PittsAM:thet}, 5.6):
\begin{proposition}\label{Lchar} The map $L:M\to {\rm Lo}$ is isomorphic (as maps of preorders) to the map
$$L'(f)(p)\; =\;\bigcap\{q\subseteq\mathbb{N}\, |\, \{ 0\}\wedge p\subseteq q\text{ and }\{ 1\}\wedge f(q)\subseteq q\}$$\end{proposition}
\begin{proof} Given $f\in M$ and $d\in E(f)$, we shall produce, recursively in $d$, elements of $[L(f)\leq L'(f)]$ and $[L'(f)\leq L(f)]$.

First, for $e\in L(f)(p)$ and indices $\alpha$ and $\beta$ such that $\alpha x=\langle 0,x\rangle$ and $\beta x=\langle 1,x\rangle$, we have: if $\{ 0\}\wedge p\subseteq q$ and $\{ 1\}\wedge f(q)\subseteq q$ then $\alpha :p\to q$ and $\beta :f(q)\to q$ hence $e\langle \alpha ,\beta\rangle\in q$. We conclude that $\lambda e.e\langle\alpha ,\beta\rangle\in [L(f)\leq L'(f)]$.

Conversely, from the interpretation in $\Eff$ of the true propositional formulas $\forall p.p\to L(f)(p)$ and $\forall p.f(L(f)(p)\to L(f)(p)$ (as we saw in the proof of~\ref{L1}) we find elements
$$\begin{array}{rcl}a & \in & \bigcap_{p\subseteq\mathbb{N}}[p\to L(f)(p)] \\
b & \in & \bigcap_{p\subseteq\mathbb{N}}[f(L(f)(p))\to L(f)(p)]\end{array}$$
Let $d\in E(f)$. By the recursion theorem, choose an index $c$ such that for all $x,y$:
$$\begin{array}{rcl}c\langle 0,x\rangle & \simeq & ax \\
c\langle 1,y\rangle & \simeq & b(dcy) \end{array}$$
Let $S\, =\, \{ z\, |\, cz\in L(f)(p)\}$. Then clearly $\{ 0\}\wedge p\subseteq S$. Moreover we have $c:S\to L(f)(p)$ hence $\lambda y.dcy:f(S)\to f(L(f)(p))$. So if $\langle 1,y\rangle\in\{ 1\}\wedge f(S)$ then $c\langle 1,y\rangle\in L(f)(p)$. We see therefore, that also $\{ 1\}\wedge f(S)\subseteq S$. By definition of $L'(f)(p)$ we have $L'(f)(p)\subseteq S$ and thus $c:L'(f)(p)\to L(f)(p)$ for all $p$, whence $c\in [L'(f)\leq L(f)]$, as desired.\end{proof}
\medskip

\noindent Let us examine some structure of the preorder $M$. $(M,\leq )$ is an internal Heyting prealgebra (a cartesian closed preorder with finite joins): finite joins and meets are given pointwise (and the constant maps to $\emptyset$ and $\mathbb{N}$ are the bottom and top elements, respectively), and Heyting implication is given by the formula
$$\begin{array}{rcl}(f\to g)(p) & = & \{\langle a,b,c\rangle\, |\,\text{there is an $h\in M$ such that }a\in E(h),\\
 & & b\in [(h\wedge f)\leq g]\text{ and }c\in h(p)\}\end{array}$$
as is easy to verify.

Next, we discuss internal joins. The preorder $(M,\leq )$ is internally cocomplete. Since any object of $\Eff$ is covered by a partitioned assembly, it suffices to consider maps into $M$ from partitioned assemblies. So, let $(X,\pi )$ and $(Y,\rho )$ be partitioned assemblies (with $\pi :X\to\mathbb{N}$, $\rho :Y\to\mathbb{N}$); let $A$ be a subobject of $(X,\pi )\times (Y,\rho )$ and $q:A\to M$ a map. The internal join along $q$, i.e.\ the map $(X,\pi )\to M$ defined internally by
$$x\mapsto\bigvee_{(x,y)\in A}q(x,y)$$
is represented by the function
$$H_A(x)\; =\; \bigcup_{y\in Y}\{\langle n,e\rangle\, |\, n\in [A(x,y)], e\in q(x,y)\}$$

We now wish to establish a connection between $M$ and a preorder structure on the sheaf $\nabla ({\cal P}{\cal P}(\mathbb{N}))$, but actually the theorem we have in mind works only if we restrict to the subassembly $M^*$ of $M$ on those functions $f$ which satisfy $\bigcup_{p\subseteq\mathbb{N}}f(p)\neq\emptyset$, and $\nabla ({\cal P}^*{\cal P}(\mathbb{N}))$ (writing ${\cal P}^*(X)$ for the set of nonempty subsets of $X$). Note that the condition defining elements of $M^*$ is always satisfied by $L(f)$, so we still have that Lo is equivalent to $(M^*,\leq _L)$.

The reader should note that in $\Eff$, $\nabla ({\cal P}(\mathbb{N}))$ is the object ${\cal P}_{\neg\neg}(N)$ of $\neg\neg$-closed subobjects of $N$, and $\nabla ({\cal P}^*{\cal P}(\mathbb{N}))$ is the object of $\neg\neg$-inhabited, $\neg\neg$-closed subobjects of  ${\cal P}_{\neg\neg}(N)$. Also, the image of $M^*$ under the projection $M\to {\rm Mon}$ is $\{ f:{\rm Mon}\, |\,\neg\neg\exists p.f(p)\}$.

For ${\cal A},{\cal B}\in {\cal P}^*{\cal P}(\mathbb{N})$ let
$$[{\cal A}\leq {\cal B}]\; =\; \{ k\, |\,\forall A\in {\cal A}\exists B\in {\cal B}(k:B\to A)\}$$
The proof of the following proposition is left to the reader.
\begin{proposition}\label{Gembed} Define a function $G_{(-)}: {\cal P}^*{\cal P}(\mathbb{N})\to {\cal P}(\mathbb{N})^{{\cal P}(\mathbb{N})}$ by
$$G_{\cal A}(p)\; =\; \bigcup_{A\in {\cal A}}(A\to p)$$
\begin{alist}\item $G_{(-)}$ is a well-defined map: $\nabla ({\cal P}^*{\cal P}(\mathbb{N}))\to M^*$ and an embedding of preorders (it preserves and reflects the order).
\item $G_{\cal A}$ is the least $f\in M^*$ such that $\bigcap_{A\in {\cal A}}f(A)$ is inhabited. That is: there are indices $b$ and $c$ such that for each ${\cal A}\in {\cal P}^*{\cal P}(\mathbb{N})$, $f\in M^*$ and $a\in E(f)$ the following hold:\begin{rlist}
\item if $x\in\bigcap_{A\in {\cal A}}f(A)$ then $b\langle a,x\rangle\in [G_{\cal A}\leq f]$
\item if $y\in [G_{\cal A}\leq f]$ then $c\langle a,y\rangle\in \bigcap_{A\in {\cal A}}f(A)$\end{rlist}
In other words, if $\pi :\nabla ({\cal P}(\mathbb{N}))\to\Omega$ is the standard surjection, then the following is internally true in $\Eff$:
$$\forall {\cal A}:\nabla ({\cal P}^*{\cal P}(\mathbb{N}))\forall f:M^*[G_{\cal A}\leq f\leftrightarrow\forall A\in {\cal A}.\pi (f(A))]$$\end{alist}\end{proposition}
\begin{theorem}\label{Mfreecom} The preorder $(M^*,\leq )$ is (internally in $\Eff$) the free completion of $(\nabla ({\cal P}^*{\cal P}(\mathbb{N})),\leq )$ under joins indexed by nonempty subsets of $N$ (where, internally, $A\subseteq N$ is `nonempty' iff $\neg\neg\exists n(n\in A)$).\end{theorem}
\begin{proof} Recall that in $\Eff$, the object of nonempty subobjects of $N$ is $\nabla ({\cal P}^*(N))$, with element relation $[n\in A]\equiv\{ n\, |\, n\in A\}$.

For $f\in M^*$, define $A\in\nabla ({\cal P}(\mathbb{N}))$ and $\theta :A\to\nabla ({\cal P}^*{\cal P}(\mathbb{N}))$ by
$$\begin{array}{rcl}A & = & \bigcup_{p\subseteq\mathbb{N}}f(p) \\
\theta (n) & = & \{ q\subseteq\mathbb{N}\, |\, n\in f(q)\} \end{array}$$
The reader can verify that $A$ and $\theta$ are well-defined. Now recall from the remark we made at the beginning of this section that $f$ is isomorphic (in the preorder $(M^*,\leq )$) to $F(f)$ where
$$F(f)(p)\; =\; \bigcup_{q\subseteq\mathbb{N}}(f(q)\wedge (q\to p))$$
From which we derive
$$F(f)(p)\; =\; \{\langle n,e\rangle\, |\, n\in A, e\in\bigcup_{q\in\theta (n)}(q\to p)\}\; = \; (\bigvee_{n\in A}G_{\theta (n)})(p)$$
So we see that $f$ is a join of a family of elements of $\nabla ({\cal P}^*{\cal P}(\mathbb{N}))$, indexed by a nonempty subset of $N$.

Next, we see that elements of the form $G_{\cal A}$, ${\cal A}\in {\cal P}^*{\cal P}(\mathbb{N})$ are {\em inaccessible for joins indexed by nonempty subsets of $N$}. That is, let $A\subseteq N$ nonempty, $h:A\to M^*$ a map. Then $G_{\cal A}\leq\bigvee_{n\in A}h_n$ implies $\exists n\in A.G_{\cal A}\leq h_n$, internally in $\Eff$. This is seen as follows:

Suppose $e\in [G_{\cal A}\leq\bigvee_{n\in A}h_n]$, so
$$e\in\bigcap_{p\subseteq\mathbb{N}}[\bigcup_{B\in {\cal A}}(B\to p)\to\{\langle n,u\rangle\, |\, n\in A, u\in h_n(p)\} ]$$
Since ${\cal A}\neq\emptyset$, there is some $B\in {\cal A}$. Let $i$ be an index for the identity function, then instantiating this $B$ for $p$ we get
$$ei\in\{\langle n,u\rangle\, |\, n\in A, u\in h_n(B)\}$$
This holds for all $B\in {\cal A}$. So we have found an $n=(ei)_0$, satisfying $(ei)_1\in h_n(B)$ for all $B\in {\cal A}$. 

Since $h:A\to M^*$ is a map, from $n$ we find some element $a_n\in E(h_n)$.

Now if $d:B\to p$ is arbitrary, $B\in {\cal A}$, $p\subseteq\mathbb{N}$, then $a_nd:h_n(B)\to h_n(p)$, hence $(a_nd)(ei)_1\in h_n(p)$. We see that for all $p$,
$$\lambda d.(a_nd)(ei)_1:(\bigcup_{B\in {\cal A}}(B\to p))\to h_n(p)$$
which means $G_{\cal A}\leq h_n$, as desired.

The two properties together imply, constructively, that $M^*$ is the stated free completion.

Indeed, suppose $(P,\leq )$ is an internal preorder in $\Eff$ which has joins indexed by nonempty subsets of $N$, and $w:(\nabla ({\cal P}^*{\cal P}(\mathbb{N})),\leq )\to P$ is order-preserving. Then we extend $w$ uniquely to a map $W:M^*\to P$ which preserves joins indexed by nonempty subsets of $N$: for $f\in M^*$, express $f$ as $\bigvee_{n\in A}\theta (n)$. Define $W(f)=\bigvee_{n\in A}w(\theta (n))$. Use the inaccessibility property to show that $W$ is well-defined.\end{proof}
\medskip

\noindent In view of Theorem~\ref{Mfreecom} we shall call elements of $M^*$ of the form $G_{\cal A}$ {\em basic}; and we shall call local operators of the form $L(G_{\cal A})$ also basic.
\section{Known results about local operators in $\Eff$}
In this section we collect some results which have appeared in the literature, as far as relevant for this paper.

The top element of Loc, the function constant $\top$, is the local operator whose category of sheaves is the trivial topos; hence this local operator will also be called {\em trivial}. The least element of Loc, the identity map on $\Omega$, will be denoted {\sf id}.

As is well-known from \cite{HylandJ:efft}, there is a largest nontrivial local operator. This is the {\em double negation\/} operator $\neg\neg$: the function sending $\emptyset$ to $\emptyset$ and everything else to $\mathbb{N}$.
\begin{proposition}[Hyland-Pitts]\label{trivnotnot}~~~~~~~~~~~~~~~~\begin{rlist}
\item Let $j\in M$. Then $L(j)$ represents the trivial local operator if and only if $j(\emptyset )\neq\emptyset$.
\item Let $j\in M$. Then $L(j)$ represents the $\neg\neg$-operator if and only if either of the following equivalent conditions holds:\begin{alist}
\item $j(\emptyset )=\emptyset$ and $\bigcap_{p\neq\emptyset}L(j)(p)\neq\emptyset$
\item $j(\emptyset )=\emptyset$ and $\bigcap_{n\in\mathbb{N}}L(j)(\{ n\} )\neq\emptyset$
\item $j(\emptyset )=\emptyset$ and $L(j)(\{ 0\} )\cap L(j)(\{ 1\} )\neq\emptyset$\end{alist}
\item Let ${\cal A}\in {\cal P}^*{\cal P}(\mathbb{N})$. Then ${\sf id}<L(G_{\cal A})$ if and only if $\bigcap {\cal A}=\emptyset$\end{rlist}\end{proposition}
We conclude that the identity, the trivial local operator and the $\neg\neg$-operator are basic: the identity is $L(G_{\{\{ 0\}\}})$, the trivial one is $L(G_{\{\emptyset\} })$ and $\neg\neg$ is $L(G_{\{\{ 0\} ,\{ 1\}\} })=L(G_{\{ p\subseteq\mathbb{N}|p\neq\emptyset\}})$.

The following corollary is easy.
\begin{corollary}\label{disjointcor} Suppose ${\cal A}\in {\cal P}^*{\cal P}(\mathbb{N})$ contains two r.e.\ separable sets, that is: sets $A_1$ and $A_2$ such that for two disjoint recursively enumerable sets $C,D$ we have $A_1\subseteq C$, $A_2\subseteq D$. Then $\neg\neg\leq L(G_{\cal A})$.\end{corollary}

A different basic local operator was identified by Pitts in \cite{PittsAM:thet}, 5.8:
\begin{proposition}[Pitts]\label{Pittsloc} Let ${\cal A}\, =\, \{\{ m\, |\, m\geq n\}\, |\, n\in\mathbb{N}\}$. Then $L(G_{\cal A})$ is strictly between {\sf id} and $\neg\neg$.\end{proposition}
Examples of non-basic local operators are those which force a partial function to be total. Suppose $f:\mathbb{N}\to\mathbb{N}$ is a function. The $\neg\neg$-closed subobject of $N\times N$ in $\Eff$ given by $\{ (n,f(n))\, |\, n\in\mathbb{N}\}$ is a single-valued relation whose domain $D_f$  is a $\neg\neg$-dense subobject of $N$. The least local operator which forces $D_f$ to be the whole of $N$ is $L(\bigvee_nG_{\rho (n)})$ where $\rho (n)=\{\{ f(n)\}\}$. Recall that $\bigvee_nG_{\rho (n)}(p)=\{\langle n,e\rangle\, |\, ef(n)\in p\}$
\begin{theorem}[Hyland]\label{HylandTuring} Denoting this least local operator by $j_f$, we have $j_f\leq j_g$ if and only if $f$ is Turing reducible to $g$.\end{theorem}
The following proposition is due to Phoa (\cite{PhoaW:relcet}):
\begin{proposition}[Phoa]\label{phoastelling} If $j$ is a local operator such that $j_f\leq j$ for each $f:\mathbb{N}\to\mathbb{N}$, then $\neg\neg\leq j$.\end{proposition}
In general, if $X\stackrel{m}{\to}Y$ is a monomorphism in $\Eff$ there is (by standard topos theory) a least local operator $j$ for which $m$ is dense. Let us write this out explicitly for the case that $Y$ is an assembly (since every object of $\Eff$ is covered by an assembly, this covers the general case): let $Y=(Y,E)$ and $R:Y\to {\cal P}(\mathbb{N})$ be such that $\bigcap_{y\in Y}(R(y)\to E(y))$ is nonempty, representing the subobject $m$. Then the least local operator for which $m$ is dense is $L(\bigvee_nG_{\theta (n)})$ where $\theta (n)=\{ R(y)\, |\, n\in E(y)\}$.

Another non-basic local operator in $\Eff$ is described in \cite{OostenJ:extlrh,OostenJ:remlrt}. Let Tot be the set of indices of total recursive functions. Consider the assembly $A=(A,E)$ where 
$$\begin{array}{rcl}A & = & \{ \langle e,f\rangle\, |\, e,f\in {\rm Tot}\text{ and }\forall nm(en=0\vee fm=0)\} \\
E(\langle e,f\rangle ) & = & \{\langle e,f\rangle\} \end{array}$$
Let $R:A\to {\cal P}(\mathbb{N})$ send $\langle e,f\rangle$ to the set
$$\{\langle e,f,0\rangle\, |\,\forall n(en=0)\}\cup\{\langle e,f,1\rangle\, |\,\forall m(fm=0\}$$
Then $R$ determines a subobject $[R]$ of $A$ and let $j_L$ be the least local operator for which this inclusion is dense.

The local operator $j_L$ corresponds to the Lifschitz subtopos of $\Eff$. In \cite{OostenJ:remlrt} it is proved that $j_L$ is the least local operator for which the following principle of first-order arithmetic, there called $B\Sigma ^0_1{\rm -MP}$ is true in the corresponding sheaf subtopos:
$$\forall e(\neg\neg\exists n(n\in [e])\to\exists n(n\in [e]))$$
where $[e]$ denotes $\{ n\leq (e)_1\, |\, (e)_0n\uparrow\}$. It can be shown that $B\Sigma ^0_1{\rm -MP}$ is equivalent to the "Lesser Limited Principle of Omniscience", which has some standing in generalized computability and constructive analysis (see e.g.\ \cite{BrattkaV:weidop,IshiharaH:omnpkl}). Since decidability of the Halting Problem implies this principle, we conclude that $j_L\leq j_h$, if $h$ is a decision function for the Halting Problem. In fact we have $j_L<j_h$, since the Halting Problem is not decidable in the Lifschitz topos.
\section{Sights}
In this section we develop some theory of a certain type of well-founded trees, which we call {\em sights}, which will enable us to derive inequalities and non-inequalities between a number of new local operators in $\Eff$. The basic insight is that elements of $L(f)(p)$ are functions defined by recursion over a well-founded tree (see in particular definition~\ref{supportingdef} and the discussion preceding it,  and proposition~\ref{dedsup}).

Let us look again at the operator $L'$ from Proposition~\ref{Lchar}:
$$L'(f)(p)\; =\;\bigcap\{ q\subseteq\mathbb{N}\, |\,\{ 0\}\wedge p\subseteq q\text{ and }\{ 1\}\wedge f(q)\subseteq q\}$$
for $f\in M$.

We can present $L'$ also in the following way:
\begin{proposition}\label{Lordinal} For ordinals $\alpha <\omega _1$, define $L(f)(p)_{\alpha}$ as follows:
$$\begin{array}{rcl}L(f)(p)_0 & = & \{ 0\}\wedge p \\
L(f)(p)_{\alpha +1} & = & L(f)(p)_{\alpha}\,\cup\, (\{ 1\}\wedge f(L(f)(p)_{\alpha}) )\\
L(f)(p)_{\lambda} & = & \bigcup_{\beta <\lambda}L(f)(p)_{\beta}\text{ for limit }\lambda \end{array}$$
Then $L'(f)(p)=L(f)(p)_{\omega _1}$. Of course, since $L'(f)(p)$ is a countable set, there is a countable ordinal $\alpha$ such that $L'(f)(p)=L(f)(p)_{\alpha}$.\end{proposition}
Proposition~\ref{Lordinal} leads us to the following definition.
\begin{definition}\label{sightsdef}\em A {\em sight\/} is, inductively,\begin{itemize}
\item[] either a thing called {\sf NIL},
\item[] or a pair $(A,\sigma )$ where $A\subseteq\mathbb{N}$ and $\sigma$ a function on $A$ such that $\sigma (a)$ is a sight for each $a\in A$.\end{itemize}\end{definition}
Let $\theta$ be a function $B\to {\cal P}^*{\cal P}(\mathbb{N})$ for $B\subseteq\mathbb{N}$ nonempty. With $\theta$ we associate (as in~\ref{Mfreecom}) the element $G_{\theta}$ of $M^*$ given by
$$G_{\theta}(p)\; =\; \{\langle n,e\rangle\, |\, n\in B,\exists A\in\theta (n)(e:A\to p)\}$$
So, $G_{\theta}=\bigvee_{n\in B} G_{\theta (n)}$.
\begin{definition}\label{dedicateddef}\em For $\theta$ as above, $p\subseteq\mathbb{N}$ and $z\in\mathbb{N}$ we say that a sight $S$ is $(z,\theta ,p)$-{\em dedicated\/} if\begin{itemize}
\item[] either $S={\sf NIL}$ and $z\in\{ 0\}\wedge p$,
\item[] or $S=(A,\sigma )$, $z=\langle 1,\langle n,e\rangle\rangle$, $A\in\theta (n)$, and for all $a\in A$, $ea$ is defined and $\sigma (a)$ is $(ea,\theta ,p)$-dedicated.\end{itemize}\end{definition}
\begin{proposition}\label{Lsightschar} For $\theta$, $z$, $p$ as before, we have:

$z\in L'(G_{\theta})(p)$ if and only if there is a $(z,\theta ,p)$-dedicated sight.\end{proposition}
\begin{proof} We use~\ref{Lordinal}. First we prove that for each $\alpha <\omega _1$, if $z\in L(G_{\theta})(p)_{\alpha}$ then there is a $(z,\theta ,p)$-dedicated sight.

For $\alpha =0$: if $z\in L(G_{\theta})(p)_0 = \{ 0\}\wedge p$, then {\sf NIL} is $(z,\theta ,p)$-dedicated.

For $\alpha +1$: suppose $z\in L(G_{\theta})(p)_{\alpha +1}$. By induction hypothesis we may assume $z\in\{ 1\}\wedge G_{\theta}(L(G_{\theta })(p)_{\alpha})$. Then $z=\langle 1,\langle n,e\rangle\rangle$ and for some $A\in\theta (n)$ we have $e:A\to L(G_{\theta})(p)_{\alpha}$. By induction hypothesis, for each $a\in A$ there is an $(ea,\theta ,p)$-dedicated sight $\sigma (a)$. Then $(A,\sigma )$ is $(z,\theta ,p)$-dedicated.

The case of limit ordinals is obvious.

Conversely, suppose that $S$ is a $(z,\theta ,p)$-dedicated sight. If $S={\sf NIL}$, then $z\in\{ 0\}\wedge p$ so $z\in L(G_{\theta})(p)_0$. If $S=(A,\sigma )$ then $z=\langle 1,\langle n,e\rangle\rangle$ and for some $A\in\theta (n)$, $\sigma (a)$ is $(ea,\theta ,p)$-dedicated for each $a\in A$. By induction hypothesis, for each $a\in A$ there is some $\alpha _a<\omega _1$ such that $ea\in L(G_{\theta})(p)_{\alpha _a}$. Then $z\in L(G_{\theta})(p)_{\beta}$ where $\beta =(\bigcup_{a\in A}\alpha _a)+1$, as is easy to see.\end{proof}
\begin{corollary}\label{leqLchar} For ${\cal A}\in {\cal P}^*{\cal P}(\mathbb{N})$, $B\subseteq\mathbb{N}$ nonempty and $\theta :B\to {\cal P}^*{\cal P}(\mathbb{N})$ we have: $G_{\cal A}\leq _LG_{\theta}$ if and only if there exists a number $z$ such that for every $A\in {\cal A}$ there exists a $(z,\theta ,A)$-dedicated sight.\end{corollary}
\begin{proof} By~\ref{Gembed}, $G_{\cal A}\leq L'(G_{\theta})$ if and only if $\bigcap_{A\in {\cal A}}L'(G_{\theta})(A)$ is nonempty, which, by \ref{Lsightschar}, is equivalent to the given statement.\end{proof}
\begin{corollary}\label{leqLchar2} For $B,B'\subseteq\mathbb{N}$ nonempty,  $\theta :B\to {\cal P}^*{\cal P}(\mathbb{N})$ and $\zeta :B'\to {\cal P}^*{\cal P}(\mathbb{N})$ we have: $G_{\zeta}\leq _LG_{\theta}$ if and only if there is a partial recursive function $f$ defined on $B'$, and for every $n\in B'$ an $(f(n),\theta ,\zeta (n))$-dedicated sight.\end{corollary}
\medskip

\noindent To any sight $S$ we associate a well-founded tree ${\rm Tr}(S)$ of coded sequences of natural numbers together with a specified subset of its set of leaves (which we will call {\em good leaves}) as follows:

If $S={\sf NIL}$ then ${\rm Tr}(S)=\{\langle\rangle\}$ and $\langle\rangle$ is a good leaf of $S$.

If $S=(\emptyset ,\emptyset )$ then ${\rm Tr}(S)=\{\langle\rangle\}$ and ${\rm Tr}(S)$ has no good leaf.

If $S=(A,\sigma )$ with $A\neq\emptyset$ then ${\rm Tr}(S)=\{ \langle a\rangle {\ast} t\, |\, a\in A,t\in {\rm Tr}(\sigma (a))\}$ and $\langle a\rangle {\ast} t$ is a good leaf of ${\rm Tr}(S)$ if and only if $t$ is a good leaf of ${\rm Tr}(\sigma (a))$.

We shall often abuse language and talk about the ``(good) leaves of a sight $S$'' instead of ${\rm Tr}(S)$.

We call a sight {\em degenerate\/} if not all its leaves are good.

Given a sight $S$ and $s\in {\rm Tr}(S)$, we write ${\rm Out}(s)$ (or ${\rm Out}_S(s)$ if we wish to emphasize the sight $s$ lives in) for the set $\{ a\in\mathbb{N}\, |\, s{\ast}\langle a\rangle\in {\rm Tr}(S)\}$.
\medskip

\noindent The following proposition follows by an easy induction on sights.
\begin{proposition}\label{degenerate->empty}If a degenerate sight is $(z,\theta ,p)$-dedicated then $\emptyset\in\bigcup_n\theta (n)$.\end{proposition}
\begin{definition}\label{supportingdef}\em Let $B\subseteq\mathbb{N}$ nonempty, $\theta :B\to {\cal P}^*{\cal P}(\mathbb{N})$, $p\subseteq\mathbb{N}$. For a number $w$, we call a sight $S$ $(w,\theta ,p)$-{\em supporting\/} if \begin{itemize}\item[] whenever $s$ is a good leaf of $S$, $ws\in\{ 0\}\wedge p$
\item[] whenever $s$ is not a good leaf of $S$, $ws=\langle 1,n\rangle$ with $n\in B$ and ${\rm Out}_S(s)\in\theta (n)$\end{itemize}\end{definition}
\begin{proposition}\label{dedsup} There are partial recursive functions $F$ and $G$ such that for each $B\subseteq\mathbb{N}$ nonempty, $\theta :B\to {\cal P}^*{\cal P}(\mathbb{N})$, $p\subseteq\mathbb{N}$, sight $S$ and $z\in\mathbb{N}$:\begin{rlist}
\item If $S$ is $(z,\theta ,p)$-dedicated then $F(z)$ is defined and $S$ is $(F(z),\theta ,p)$-supporting.
\item If $S$ is $(w,\theta ,p)$-supporting then $G(w)$ is defined and $S$ is $(G(w),\theta ,p)$-dedicated.\end{rlist}\end{proposition}
\begin{proof} i) Note that from the definition of ``$S$ is $(w,\theta ,p)$-supporting'' it follows that if $H$ is a partial recursive function such that for each $a\in A$, $H(a)$ is defined and the sight $\sigma (a)$ is $(H(a),\theta ,p)$-supporting, and
$$w\; =\;\lambda s.\left\{\begin{array}{cl}\langle 1,n\rangle & \text{if }s=\langle\rangle \\
H((s)_0)\langle (s)_1,\ldots ,(s)_{{\rm lh}(s)-1}\rangle & \text{otherwise}\end{array}\right.$$
then the sight $(A,\sigma )$ is $(w,\theta ,p)$-supporting: $s$ is a good leaf of $(A,\sigma )$ if and only if $\langle (s)_1,\ldots ,(s)_{{\rm lh}(s)-1}\rangle$ is a good leaf of $\sigma ((s)_0)$.

Therefore, using the recursion theorem let $F$ be partial recursive such that
$$F(z)s\;\simeq\;\left\{\begin{array}{cl} z & \text{if }z=\langle 0,y\rangle \\
\left\{\begin{array}{cl}\langle 1,n\rangle & \text{if }s=\langle\rangle \\
F(e(s)_0)\langle (s)_1,\ldots ,(s)_{{\rm lh}(s)-1}\rangle & \text{else}\end{array}\right\} & \text{if }z=\langle 1,\langle n,e\rangle\rangle \end{array}\right.$$
The proof is now by induction on $S$: if $S={\sf NIL}$ and $S$ is $(z,\theta ,p)$-dedicated then $z=\langle 0,y\rangle$, $y\in p$, $F(z)\langle\rangle =z$ and $S$ is $(F(z),\theta ,p)$-supporting. If $S=(A,\sigma )$ is $(z,\theta ,p)$-dedicated then $z=\langle 1,\langle n,e\rangle\rangle$ etc., and for each $a\in A$ by induction hypothesis $F(e(s)_0)$ is defined and $\sigma (a)$ is $(F(e(s)_0),\theta ,p)$-supporting. By our first remark it now follows that $S=(A,\sigma )$ is $(F(z),\theta ,p)$-supporting.
\medskip

\noindent ii) Here we remark that if $A\in\theta (n)$ and for each $a\in A$, $ea$ is defined and $\sigma (a)$ is $(ea,\theta ,p)$-dedicated, then $(A,\sigma )$ is $(\langle 1,\langle n,e\rangle\rangle ,\theta ,p)$-dedicated.

Also, note that if $(A,\sigma )$ is $(w,\theta ,p)$-supporting then for each $a\in A$, $\sigma (a)$ is $(\lambda s.w(\langle a\rangle {\ast}s),\theta ,p)$-supporting.

Define $G$, using the recursion theorem, by
$$G(w)\;\simeq\;\left\{\begin{array}{cl}\langle 0,y\rangle & \text{if }w\langle\rangle =\langle 0,y\rangle \\
\langle 1,\langle n,\lambda a.G(\lambda s.w(\langle a\rangle {\ast}s))\rangle\rangle & \text{if }w\langle\rangle =\langle 1,n\rangle \end{array}\right.$$
Proof, again by induction on $S$: suppose $S$ is $(w,\theta ,p)$-supporting. If $S={\sf NIL}$ then $w\langle\rangle =\langle 0,y\rangle$, $y\in p$ and $G(w)=\langle 0,y\rangle$, so $S$ is $(G(w),\theta ,p)$-dedicated.

If $S=(A,\sigma )$ then $w\langle\rangle =\langle 1,n\rangle$ for an $n$ such that $A\in\theta (n)$. By our remark, for each $a\in A$ the sight $\sigma (a)$ is $(\lambda s.w(\langle a\rangle {\ast}s),\theta ,p)$-supporting hence by induction hypothesis, $\sigma (a)$ is $(G(\lambda s.w(\langle a\rangle {\ast}s)),\theta ,p)$-dedicated. Then if $e=\lambda a.G(\lambda s.w(\langle a\rangle {\ast}s))$, $(A,\sigma )$ is $(\langle 1,\langle n,e\rangle\rangle ,\theta ,p)$-dedicated; i.e., $(A,\sigma )$ is $(G(w),\theta ,p)$-dedicated, as desired.\end{proof}
\begin{corollary}\label{dedsupcorollary} For $\theta :B\to {\cal P}^*{\cal P}(\mathbb{N})$, the element $L'(G_{\theta})$ of $M^*$ is, in $M^*$, isomorphic to the function which sends $p\subseteq\mathbb{N}$ to
$$\{ z\in\mathbb{N}\, |\,\text{there is a $(z,\theta ,p)$-supporting sight}\}$$\end{corollary}
The following corollary shows that the local operators $j_f$ from~\ref{HylandTuring} are not basic, in fact are not majorizing any nontrivial basic local operator.
\begin{corollary}\label{Hylandnotbasic} Suppose ${\cal A}\in {\cal P}^*{\cal P}(\mathbb{N})$ and $f:\mathbb{N}\to\mathbb{N}$ a function. Let $j_f$ be the least local operator which forces $f$ to be total, as in~\ref{HylandTuring}. Then if $G_{\cal A}\leq _Lj_f$, $L(G_{\cal A})$ is the identity local operator.\end{corollary}
\begin{proof} Let $\rho _f:n\mapsto\{\{ f(n)\}\}$ be as just above~\ref{HylandTuring}, so $G_{\cal A}\leq j_f$ if and only if $G_{\cal A}\leq _L\rho _f$. First, we prove the following
\begin{quote}{\sl Claim: given $z\in\mathbb{N}$ and sights $S$ and $T$ such that both $S$ and $T$ are $(z,\rho _f,\mathbb{N})$-dedicated, then $S=T$.}\end{quote}
We prove the Claim by induction on $S$. If $S={\sf NIL}$ then $z =\langle 0,y\rangle$ for some $y$. It follows that also $T={\sf NIL}$. If $S=(A,\sigma )$ then $z =\langle 1,\langle n,e\rangle\rangle$, $A=\{ f(n)\}$ and $\sigma (f(n))$ is $(ef(n),\rho _f,\mathbb{N})$-dedicated. Similarly, $T=(\{ f(n)\} ,\tau )$ and $\tau (f(n))$ is $(ef(n),\rho _f,\mathbb{N})$-dedicated. By induction hypothesis, $\sigma (f(n))=\tau (f(n))$ whence $S=T$, as desired. This proves the Claim.
\medskip

\noindent Now suppose $G_{\cal A}\leq _L\rho _f$. By~\ref{leqLchar}, there is a number $z$ and, for each $A\in {\cal A}$, a $(z,\rho _f,A)$-dedicated sight $S_A$. By the Claim, all $S_A$ are equal, say $S$. Since $\rho _f(n)$ never contains the empty set, $S$ is nondegenerate and by~\ref{dedsup}, it is $(F(z),\rho _f, A)$-supporting for each $A\in {\cal A}$. Take any good leaf $d$ of $S$. Then $F(z)d=\langle 0,y\rangle$ with $d\in\bigcap {\cal A}$. By~\ref{trivnotnot} iii), $L(G_{\cal A})$ is the identity local operator, as claimed.\end{proof}
\begin{definition}\label{jip}\em Suppose ${\cal A}_1,\ldots ,{\cal A}_n\in {\cal P}^*{\cal P}(\mathbb{N})$. We say that the ${\cal A}_i$ have the {\em joint intersection property\/} if for all $A_1\in {\cal A}_1,\ldots ,A_n\in {\cal A}_n$, $A_1\cap\cdots\cap A_n\neq\emptyset$.

Similarly, we say that ${\cal A}\in {\cal P}^*{\cal P}(\mathbb{N}$ has the {\em $n$-intersection property\/} if for all $A_1,\ldots ,A_n\in {\cal A}$, $A_1\cap\cdots \cap A_n\neq\emptyset$.

We say that a sight $S$ is {\em on\/} ${\cal A}$ if, inductively, $S={\sf NIL}$ or $S=(A,\sigma )$, $A\in {\cal A}$ and for all $a\in A$ the sight $\sigma (a)$ is on $\cal A$. This means that for every $d\in {\rm Tr}(S)$ which is not a good leaf, ${\rm Out}_S(d)\in {\cal A}$. We say that $S$ is {\em on\/} $\theta :B\to {\cal P}^*{\cal P}(\mathbb{N})$ if $S$ is on $\bigcup_{n\in B}\theta (n)$.\end{definition}
\begin{proposition}\label{jipprop} Suppose ${\cal A}_1,\ldots ,{\cal A}_n$ have the joint intersection property. Then if $S_i$ is a sight on ${\cal A}_i$ for each $i$, there is a coded sequence $d$ such that\begin{itemize}
\item[] $d\in {\rm Tr}(S_i)$ for each $i$, and
\item[] $d$ is a good leaf of some $S_i$.\end{itemize}\end{proposition}
\begin{proof} Induction on $S_1$. If $S_1={\sf NIL}$ then we can take $\langle\rangle$ for $d$. Similarly, if $S_i={\sf NIL}$ for some $i\geq 2$ we can take $\langle\rangle$ for $d$. So assume each $S_i$ is $(A_i,\sigma _i)$. By the joint intersection property, take $a\in\bigcap_iA_i$. By the induction hypothesis, there is a $d'$ such that $d'\in {\rm Tr}(\sigma _i(a))$ for each $i$, and $d'$ is a good leaf of some $\sigma _i(a)$. Then $\langle a\rangle\ast d'$ satisfies the proposition.\end{proof}
\begin{corollary}\label{nintcor} Suppose $\cal A$ has the $n$-intersection property. Then for every $n$-tuple of sights $S_1,\ldots ,S_n$ on $\cal A$ there is a sequence $d\in\bigcap_i{\rm Tr}(S_i)$ such that $d$ is a good leaf of at least one $S_i$.\end{corollary}
\begin{definition}\label{rdefdef}\em For a sight $S$ and a number $z$, we say that $z$ is {\em $r$-defined on\/} $S$ if for some $\theta$, $S$ is $(z,\theta ,\mathbb{N})$-dedicated.\end{definition}
\begin{proposition}\label{rdefprop} Suppose $S$ and $T$ are sights and $d=\langle d_1,\ldots ,d_n\rangle$ is an element of ${\rm Tr}(S)\cap {\rm Tr}(T)$. If some $z$ is r-defined on both $S$ and $T$ and $d$ is a good leaf of $S$,  then $d$ is also a good leaf of $T$.\end{proposition}
\begin{proof} Induction on $n$. If $n=0$ then $d=\langle\rangle$, so if $d$ is a good leaf of $S$, $S={\sf NIL}$. Then $z$, being r-defined on $S$, must be $\langle 0,y\rangle$; hence, since $z$ is r-defined on $T$, $T={\sf NIL}$ and $d$ is a good leaf of $T$.

If $n>0$ then $S=(A,\sigma ), T=(B,\tau )$. Then $\langle d_2,\ldots ,d_n\rangle$ (which is $\langle\rangle$ if $n=1$) is a good leaf of $\sigma (d_1)$ and an element of ${\rm Tr}(\tau (d_1))$; by induction hypothesis $\langle d_2,\ldots ,d_n\rangle$ is a good leaf of $\tau (d_1)$ hence $d$ is a good leaf of $T$.\end{proof}
\begin{proposition}\label{nintcor2} Let ${\cal A},{\cal B}\in {\cal P}^*{\cal P}(\mathbb{N})$ and $n\geq 1$ be such that ${\cal B}$ has the $n$-intersection property whereas $\cal A$ contains sets $A_1,\ldots ,A_n$ satisfying $\bigcap_iA_i=\emptyset$. Then $G_{\cal A}\mbox{$\not\leq$}_LG_{\cal B}$.\end{proposition}
\begin{proof} Suppose $G_{\cal A}\leq _LG_{\cal B}$ and let $A_1,\ldots ,A_n\in {\cal A}$. By~\ref{leqLchar} there is a number $z$ and for each $i$ a $(z,{\cal B},A_i)$-dedicated sight $S_i$. Since $\cal B$ has the $n$-intersection property, by~\ref{nintcor} there is a coded sequence $d\in\bigcap_i{\rm Tr}(S_i)$ which is a good leaf of at least one $S_i$. Since $z$ is r-defined on each $S_i$, \ref{rdefprop} gives that $d$ is a good leaf of each $S_i$. By \ref{dedsup}, every $S_i$ is $(F(z),{\cal B},A_i)$-supporting, which means that $F(z)d=\langle 0,y\rangle$ with $y\in\bigcap_iA_i$. This holds for any $n$-tuple $A_1,\ldots ,A_n\in {\cal A}$, so we see that ${\cal A}$ has the $n$-intersection property.\end{proof}
\section{Calculations}
We are now ready to investigate some basic local operators.

Let $\alpha$ be a natural number $>1$, or $\omega$. With $\alpha$ we associate the set $\{ 1,\ldots,\alpha\}$ if $\alpha$ is a natural number, or $\mathbb{N}$ if $\alpha =\omega$. For $m\leq\alpha\leq\omega$ let
$$O^{\alpha}_m\; =\;\{ X\subseteq\alpha\, |\, |\alpha -X|=m\}$$
the set of `co-$m$-tons' in $\alpha$. Via the map $G_{(-)}$ of~\ref{Gembed} we regard the $O^{\alpha}_m$ as elements of $M^*$ (and we write $O^{\alpha}_m$ instead of $G_{O^{\alpha}_m}$). Of course, we are really interested in the local operators generated by the $O^{\alpha}_m$, and therefore we first get some trivial cases out of the way: if $\alpha =m$ so $O^{\alpha}_m=\{\emptyset\}$, then $L(O^{\alpha}_m)$ is the trivial local operator, and if $m<\alpha\leq 2m$ then $O^{\alpha}_m$ contains two disjoint finite sets whence $\neg\neg\leq _LO^{\alpha}_m$ by~\ref{disjointcor}.

Henceforth we concentrate on the case $1<2m<\alpha\leq\omega$.
\begin{proposition}\label{Om<Om+1} Let $1<2m<\alpha\leq\omega$. Then $O^{\alpha}_m<O^{\alpha}_{m+1}$ in $M^*$.\end{proposition}
\begin{proof} For $\leq$ we need a $k$ such that for each $A\in O^{\alpha}_m$ there is $B\in O^{\alpha}_{m+1}$ with $k\in B\to A$; but we can take $\lambda x.x$ for $k$.

For $O^{\alpha}_{m+1}\mbox{$\not\leq$}O^{\alpha}_m$, suppose $k$ is such that for each $A\in O^{\alpha}_{m+1}$ there is $B\in O^{\alpha}_m$ with $k\in B\to A$. Let $\gamma $ be the restriction of the partial function $\varphi _k$ to $\alpha$ and let $C=\gamma [\alpha ]\cap\alpha$. If $|C|\leq m$ then since $2m+1\leq\alpha$ we can find an $A\in O^{\alpha}_{m+1}$ such that $C\cap A=\emptyset$, but then clearly there is no $B\in O^{\alpha}_m$ with $k\in B\to A$. So pick $m+1$ distinct elements $v_1,\ldots ,v_{m+1}\in C$. By choice of $k$ there is $B\in O^{\alpha}_m$ such that $k:B\to (\alpha -\{ v_1,\ldots ,v_{m+1}\} )$. Then we must have $\gamma [\alpha -B]=\{ v_1,\ldots ,v_{m+1}\}$ but this is impossible, since $|\gamma [\alpha -B]|\leq |\alpha -B|=m$.\end{proof}
\begin{proposition}\label{O1<LOm} Let $1\leq m<\omega$. Then $O^{\omega}_1\cong _LO^{\omega}_m$.\end{proposition}
\begin{proof} We have $O^{\omega}_1\leq O^{\omega}_m$ in $M^*$ hence $O^{\omega}_1\leq _LO^{\omega}_m$; this is left to the reader.

For the converse inequality $O^{\omega}_m\leq _LO^{\omega}_1$ we have to find (by~\ref{leqLchar} and \ref{dedsup}) a number $z$ and, for each $A\in O^{\omega}_m$, a $(z,O^{\omega}_1,A)$-supporting sight. In order to conform to definition~\ref{supportingdef} we regard $O^{\omega}_1$ as function $\{ 0\}\to {\cal P}^*{\cal P}(\mathbb{N})$ with value $O^{\omega}_1$.

Given distinct $a_1,\ldots ,a_m\in\mathbb{N}$ define
$$T_{a_1,\ldots ,a_m}\; =\; \{\langle c_1,\ldots ,c_p\rangle\, |\, p\leq m\text{ and for all }i\leq p,\, c_i\neq (a_i)_i\}$$
and let $S_{a_1,\ldots ,a_m}$ be the unique non-degenerate sight with ${\rm Tr}(S_{a_1,\ldots ,a_m})=T_{a_1,\ldots ,a_m}$.

Let $z$ be such that for each coded sequence $\langle c_1,\ldots ,c_p\rangle$,
$$z\langle c_1,\ldots ,c_p\rangle\; =\;\left\{\begin{array}{cl}\langle 1,0\rangle & \text{if }p<m \\
\langle 0,\langle c_1,\ldots ,c_m\rangle\rangle & \text{if }p\geq m\end{array}\right.$$
We claim that $S_{a_1,\ldots ,a_m}$ is $(z,O^{\omega}_1,\mathbb{N}-\{ a_1,\ldots ,a_m\} )$-supporting.

Note that for each $\langle c_1,\ldots ,c_p\rangle\in {\rm Tr}(T_{a_1,\ldots ,a_m})$ which is not a leaf, we have
$${\rm Out}(\langle c_1,\ldots ,c_p\rangle )\; =\; \{ c_{p+1}\, |\, c_{p+1}\neq (a_{p+1})_{p+1}\}$$
and this is an element of $O^{\omega}_1$. In this case, $z\langle c_1,\ldots ,c_p\rangle =\langle 1,0\rangle$ as required. If  $\langle c_1,\ldots ,c_p\rangle\in {\rm Tr}(T_{a_1,\ldots ,a_m})$ is a leaf, then $p=m$, so
$$z\langle c_1,\ldots ,c_p\rangle\; =\; \langle 0,\langle c_1,\ldots ,c_m\rangle\rangle$$
We need to see that $\langle c_1,\ldots ,c_p\rangle$ is not an element of $\{ a_1,\ldots ,a_m\}$; but this follows readily from the definition of $T_{a_1,\ldots ,a_m}$.\end{proof}
\begin{proposition}\label{ceiling} Let $1\leq m<\alpha <\omega$. Then $\ulcorner\frac{\alpha}{m}\urcorner$, the least integer $\geq\frac{\alpha}{m}$, is the least number $d$ for which there are $d$ elements $A_1,\ldots ,A_d$ of $O^{\alpha}_m$ with $\bigcap_{i=1}^dA_i=\emptyset$.\end{proposition}
\begin{proof} For any $d\geq 1$ we have: $\forall A_1,\ldots ,A_d\in O^{\alpha}_m (\bigcap_{i=1}^dA_i\neq\emptyset )$ if and only if $\forall A_1,\ldots ,A_d\in O^{\alpha}_{\alpha -m} (\bigcup_{i=1}^dA_i\neq\alpha )$ if and only if $dm<\alpha$.\end{proof}
\begin{proposition}\label{m+1notLm} Let $1<2m<\alpha <\omega$. Suppose $\ulcorner\frac{\alpha}{m+1}\urcorner <\ulcorner\frac{\alpha}{m}\urcorner$. Then $O^{\alpha}_{m+1}\mbox{$\not\leq$}_LO^{\alpha}_m$, so $O^{\alpha}_m<_LO^{\alpha}_{m+1}$.\end{proposition}
\begin{proof} Let $d=\ulcorner\frac{\alpha}{m+1}\urcorner$. Then $O^{\alpha}_{m+1}$ contains $d$ sets with empty intersection, whereas $O^{\alpha}_m$ has the $d$-intersection property. The result follows from proposition~\ref{nintcor2}.\end{proof}
\medskip

\noindent {\bf Open Problem}. We have not been able to determine whether it can happen that $O^{\alpha}_{m+1}\leq _LO^{\alpha}_m$ in the case that $\ulcorner\frac{\alpha}{m+1}\urcorner = \ulcorner\frac{\alpha}{m}\urcorner$.
\medskip

\noindent The following proposition shows that, in the preorder of basic local operators (i.e., the preorder $({\cal P}^*{\cal P}(\mathbb{N}),\leq _L)$), $O^{\omega}_1$ is an atom and $\neg\neg$ is a co-atom:
\begin{proposition}\label{omegaatom}~~~~~~~~~~~~~~~~~~~~~~\begin{rlist}
\item ${\sf id}<_LO^{\omega}_1$
\item For every ${\cal A}\in {\cal P}^*{\cal P}(\mathbb{N})$, either ${\cal A}\cong _L{\sf id}$, or ${\cal A}\cong _L\top$ (the trivial local operator), or $O^{\omega}_1\leq _L{\cal A}\leq _L\neg\neg$\end{rlist}\end{proposition}
\begin{proof} Part i) follows directly from \ref{trivnotnot}iii).

For ii): again using \ref{trivnotnot}iii), ${\cal A}\cong _L{\sf id}$ if and only if $\bigcap {\cal A}\neq\emptyset$. If $\bigcap {\cal A}=\emptyset$ then for each $n\in\mathbb{N}$ there is an $A\in {\cal A}$ with $n\mbox{$\not\in$}A$, hence $\lambda x.x\in [O^{\omega}_1\leq {\cal A}]$.

From the same proposition, part i), it follows that ${\cal A}\cong _L\top$ if and only if $\emptyset\in {\cal A}$. If $\emptyset\mbox{$\not\in$}{\cal A}$ then ${\cal A}\leq \{ p\subseteq\mathbb{N}\, |\, p\neq\emptyset\}$, so ${\cal A}\leq _L\neg\neg$.\end{proof}
\medskip

\noindent {\bf Remark}. Note that we do {\em not\/} have in $M^*$ that if ${\sf id}<f$ then $O^{\omega}_1\leq f$, as \ref{Hylandnotbasic} showed.
\begin{proposition}\label{alpha<beta} Let $1<2m<\beta\leq\alpha\leq\omega$. Then $O^{\alpha}_m\leq O^{\beta}_m$ in $M^*$.\end{proposition}
\begin{proof} Realized by $\lambda x.x$.\end{proof}
\begin{proposition}\label{alphanotleqomega} Let $1<2m<\alpha <\omega$. Then $O^{\alpha}_m\mbox{$\not\leq$}_LO^{\omega}_1$, hence $O^{\omega}_1<_LO^{\alpha}_m$.\end{proposition}
\begin{proof} Immediate from \ref{nintcor2} and \ref{omegaatom}.\end{proof}
\begin{proposition}\label{alpha+m<alpha} Let $1<2m,\alpha <\omega$. Then $O^{\alpha}_m\mbox{$\not\leq$}_LO^{\alpha +m}_m$, hence $O^{\alpha +m}_m<_LO^{\alpha}_m$.\end{proposition}
\begin{proof} Let $d=\ulcorner\frac{\alpha}{m}\urcorner$. Then $O^{\alpha}_m$ contains $d$ sets with empty intersection whereas $O^{\alpha +m}_m$ has the $d$-intersection property ($\ulcorner\frac{\alpha +m}{m}\urcorner =d+1$), so the first statement follows from \ref{nintcor2}. The second statement follows from \ref{alpha<beta}.\end{proof}
\medskip

\noindent {\bf Open Problems} 1. We do not know whether $O^{\alpha +1}_m<_LO^{\alpha}_m$.\\
2. How do, e.g., $O^{2m+1}_m$ and $O^{2n+1}_n$ compare?
\medskip

\noindent The following theorem shows that the local operators $O^{\alpha}_m$ do not create any new functions $N\to N$. Equivalently, they do not force any subobjects of $N$ to be decidable.
\begin{theorem}\label{alphanoteffective} Let $D\subseteq\mathbb{N}$ and $1<2m<\alpha\leq\omega$. Let $\chi _D$ be the characteristic function of $D$ and let $\rho _D(n)=\{\{\chi _D(n)\}\}$ (so $L(\rho _D)$ is the least local operator forcing $D$ to be decidable). We have: if $\rho _D\leq _LO^{\alpha}_m$ then $D$ is recursive.\end{theorem}
\begin{proof} Note that $\rho _D\leq _LO^{\alpha}_m$ if and only if there is a total recursive function $\zeta$ such that for all $n$ there is a $(\zeta (n),O^{\alpha}_m,\{\chi _D(n)\} )$-dedicated sight.

So let $\zeta$ be such a function. By the definition of `dedicated' it follows that for all $n$, $\zeta (n)$ is of the form $\langle i,x\rangle$ with $i\in\{ 0,1\}$; and if $i=1$, then $x=\langle n,e\rangle$.

By the recursion theorem, let $f$ be an index such that;\begin{rlist}
\item $f\langle 0,x\rangle =x$
\item for $f\langle 1,\langle n,e\rangle\rangle$, search for the least computation witnessing that there are $m+1$ distinct elements $a_1,\ldots ,a_{m+1}\in\alpha$ such that $ea_1,\ldots ,ea_{m+1}$ are all defined and moreover,
$$f(ea_1)=\cdots =f(ea_{m+1})$$ If this is found, put  $f\langle 1,\langle n,e\rangle\rangle =f(ea_1)$; if not, $f\langle 1,\langle n,e\rangle\rangle$ is undefined.\end{rlist}
We claim that the index $f$ has the following property:\begin{itemize}
\item[$(S)$] For every $\langle i,x\rangle\in\mathbb{N}$ and every $(\langle i,x\rangle ,O^{\alpha}_m,\{\chi _D(n)\} )$-dedicated sight $S$, $f\langle i,x\rangle =\chi _D(n)$\end{itemize}
Note that this implies the statement in the theorem: for all $n$ we have $f(\zeta(n))=\chi _D(n)$, which means that $D$ is recursive.

So it suffices to prove the claim $(S)$, which we do by induction on the sight $S$. If $S={\sf NIL}$ and $S$ is $(\langle i,x\rangle ,O^{\alpha}_m,\{\chi _D(n)\} )$-dedicated, then $i=0$ and $x=\chi _D(n)$; and $f\langle i,x\rangle =x=\chi _D(n)$.

Suppose $S=(A,\sigma)$ with $A\in O^{\alpha}_m$. Then $\langle i,x\rangle =\langle 1,\langle n,e\rangle\rangle$, $ea$ is defined for all $a\in A$, and $\sigma (a)$ is $(ea,O^{\alpha}_m,\{\chi _D(n)\} )$-dedicated. By induction hypothesis, for each $a\in A$ we have $f(ea)=\chi _D(n)$. There are at least $m+1$ elements in $A$ since $2m<\alpha$. So the search in part ii) of the definition of the index $f$ succeeds. And because every subset of $\alpha$ of cardinality $m+1$ intersects $A$ ($A\in O^{\alpha}_m$), we have $f\langle i,x\rangle =\chi _D(n)$.

This proves the claim and finishes the proof of the theorem.\end{proof} 
\medskip

\noindent For our next array of results, we need some more definitions about sights.
\begin{definition}\label{sectordef}\em ~~~~~~~~~~~~~~~~\begin{rlist}
\item Given a sight $S$, a {\em sector\/} of $S$ is a sight $T$ such that:\begin{alist}
\item for some subset $A$ of the set of leaves of ${\rm Tr}(S)$,
$${\rm Tr}(T)\; =\;\{ s\in {\rm Tr}(S)\, |\,\text{$s$ is an initial segment of some $t\in A$}\}$$
\item $s$ is a good leaf of $T$ if and only if $s$ is a good leaf of $S$.\end{alist}
\item We call a sight $S$ {\em finitary\/} ($n$-{\em ary}, respectively) if ${\rm Tr}(S)$ is a finitely branching ($n$-ary branching) tree.
\item If $z$ is r-defined on a sight $S$ (see \ref{rdefdef}), we write $z[S]$ for the set
$$\{ y\, |\,\text{for some $s\in {\rm Tr}(S)$, $F(z)s=\langle 0,y\rangle$}\}$$
where $F$ is the function from \ref{dedsup}. So if $S$ is $(z,\theta ,p)$-dedicated, we have $z[S]\subseteq p$.\end{rlist}\end{definition}
We are now going to have a closer look at Pitts' local operator: the operator induced by $\{\{ m\, |\, m\geq n\}\, |\, n\in\mathbb{N}\}$ given in~\ref{Pittsloc}. It is easy to see that this family of subsets of $\mathbb{N}$ is, in $({\cal P}^*{\cal P}(\mathbb{N}),\leq )$, isomorphic to the family ${\cal F}$ of cofinite subsets of $\mathbb{N}$.
\begin{proposition}\label{Fnotalpha} Let $1<2m<\alpha <\omega$. Then $\cal F$ and $O^{\alpha}_m$ are incomparable w.r.t.\ the order $\leq _L$. Moreover, ${\cal F}\mbox{$\not\leq$}_LO^{\omega}_1$.\end{proposition}
Recall that for $\alpha =\omega$ we have $O^{\omega}_m\cong _LO^{\omega}_1\leq _L{\cal F}$ by~\ref{O1<LOm} and \ref{omegaatom}.

\noindent\begin{proof} Suppose ${\cal F}\leq _LO^{\alpha}_m$ for $1<2m<\alpha\leq\omega$. Choose $z$ such that for every cofinite $X$ there is a $(z,O^{\alpha}_m,X)$-dedicated sight. Pick such a sight for $X=\mathbb{N}$, say $S$. Since every element of $O^{\alpha}_m$ has at least $m+1$ elements, $S$ has an $(m+1)$-ary sector $S'$. Then $S'$ is $(z,\{\text{the $m+1$-tons $\subset\alpha$}\} ,\mathbb{N})$-dedicated, and $S'$ is finite by K\"onig's Lemma, so $z[S']$ is finite.

Now choose a $(z,O^{\alpha}_m,\mathbb{N}-z[S'])$-dedicated sight $T$. Since:\begin{itemize}
\item[] the sight $S'$ is on $\{\text{the $m+1$-tons $\subset\alpha$}\}$
\item[] the sight $T$ is on $O^{\alpha}_m$
\item[] $\{\text{the $m+1$-tons $\subset\alpha$}\}$ and $O^{\alpha}_m$ have the joint intersection property\end{itemize}
by~\ref{jipprop} there is a coded sequence $d$ which is an element of ${\rm Tr}(S')\cap {\rm Tr}(T)$ and a good leaf of one of them; but since $z$ is r-defined on both $S'$ and $T$, by \ref{rdefprop} $d$ is a good leaf of both of them. But now we get a contradiction: $F(z)d\in z[S']\cap z[T]\subseteq z[S']\cap (\mathbb{N}-z[S'])$.

For the converse inequality (in the case $\alpha <\omega$ we simply note that $\bigcap O^{\alpha}_m=\emptyset$ and that ${\cal F}$ has the $|O^{\alpha}_m|$-intersection property. So $O^{\alpha}_m\mbox{$\not\leq$}_L{\cal F}$ by \ref{nintcor2}.\end{proof}
\medskip

\noindent We now turn to joins in $(M^*,\leq )$ and $(M^*,\leq _L)$. Joins in $(M^*,\leq )$ are easy and follow from the discussion after \ref{Lchar} and theorem~\ref{Mfreecom}: given $\theta ,\zeta :\mathbb{N}\to {\cal P}{\cal P}(\mathbb{N})$, the join $\theta\vee\zeta$ can be given as the map which sends $2n$ to $\theta (n)$ and $2n+1$ to $\zeta (n)$. Of course, the map $L$, being a left adjoint, preserves joins. However, for ${\cal A},{\cal B}\in {\cal P}^*{\cal P}^*(\mathbb{N})$ there is a simpler description of their join w.r.t.\ $\leq _L$, which also makes clear that the join is a basic local operator.

We shall write $\vee _L$ for the join w.r.t. $\leq _L$. Define
$${\cal A}\odot {\cal B}\; =\; \{ A\wedge B\, |\, A\in {\cal A},B\in {\cal B}\}$$
\begin{proposition}\label{joinprop} For ${\cal A},{\cal B}\in {\cal P}^*{\cal P}^*(\mathbb{N})$, the join ${\cal A}\vee _L{\cal B}$ is given by ${\cal A}\odot {\cal B}$.\end{proposition}
\begin{proof} It is easy that ${\cal A}\leq {\cal A}\odot {\cal B}$ hence also $\leq _L$; and, of course, the same for $\cal B$. If ${\cal A},{\cal B}\leq _Lf$ so ${\cal A},{\cal B}\leq L(f)$ we have $a\in\bigcap_{A\in {\cal A}}L(f)(A)$, $b\in\bigcap_{B\in {\cal B}}L(f)(B)$ which, using that $L(f)$ is a local operator, gives an element of
$$\bigcap_{A\in {\cal A},B\in {\cal B}}L(f)(A\wedge B)$$ 
which means that ${\cal A}\odot {\cal B}\leq L(f)$.\end{proof}
\begin{proposition}\label{nintand} Suppose ${\cal A}_1,\ldots ,{\cal A}_k\in {\cal P}^*{\cal P}^*(\mathbb{N})$ such that each ${\cal A}_i$ has the $n_i$-intersection property. Then ${\cal A}_1\odot\cdots\odot {\cal A}_k$ has the $m$-intersection property if and only if $m\leq\min\{ n_1,\ldots ,n_k\}$.\end{proposition}
\begin{proof} In one direction, use induction on $k$; in the other, observe that if some ${\cal A}_i$ does not have the $m$-intersection property, then ${\cal A}_1\odot\cdots\odot {\cal A}_k$ cannot have it.\end{proof}
\begin{proposition}\label{alphaofFnotnot} Let $1<2m<\alpha\leq\omega$. Then $O^{\alpha}_m\vee _L{\cal F}<_L\neg\neg$.\end{proposition}
\begin{proof} It is left to the reader that $O^{\alpha}_m\odot {\cal F}\leq \neg\neg$. To prove that $\neg\neg\mbox{$\not\leq$}_LO^{\alpha}_m\odot {\cal F}$, observe that $\neg\neg =L(\{\{ 0\} ,\{ 1\}\} )$ and that $O^{\alpha}_m\odot {\cal F}$ has, by \ref{nintand} the $2$-intersection property; so \ref{nintcor2} can be applied.\end{proof}
\begin{proposition}\label{of2mm<notnot} Let $1\leq k\in\mathbb{N}$. Then $(\bigvee_{1\leq m\leq k} )_LO^{2m+1}_m<_L\neg\neg$.\end{proposition}
\begin{proof} By \ref{nintand}, $\bigodot_{1\leq m\leq k}O^{2m+1}_m$ has the 2-intersection property, so again by \ref{nintcor2} we have $\neg\neg\mbox{$\not\leq$}_L\bigodot_{1\leq m\leq k}O^{2m+1}_m$.\end{proof}
\begin{proposition}\label{of2mmnot<F} Let $1\leq k\in\mathbb{N}$. Then $(\bigvee_{1\leq m\leq k})_L\mbox{$\not\leq$}{\cal F}$.\end{proposition}
\begin{proof} $O^3_1$ does not have the 3-intersection property. Apply \ref{nintand} and \ref{nintcor2}.\end{proof}
\medskip

\noindent {\bf Open Problem}. One might be able to mimic (the proof of) \ref{Fnotalpha} to show that $${\cal F}\mbox{$\not\leq$}_L\bigodot_{1\leq m\leq k}O^{2m+1}_m$$
We have not been able to carry this out, however.
\medskip

\noindent We conclude with a theorem saying that Pitts' local operator $L({\cal F})$ forces every arithmetically definable set of numbers to be decidable. This implies that the subtopos of $\Eff$ corresponding to this local operator, although not a Boolean topos, nevertheless satisfies true arithmetic, as will be proved in \ref{pitts=true}. First a lemma:
\begin{lemma}\label{Pittsarithmlemma} Let $j$ be a local operator. Then for every recursive function $F$, acting on coded sequences, we have a partial recursive function $G$ (obtained uniformly in $F$) such that for each $n$, each coded sequence $\sigma =\langle\sigma _0,\ldots ,\sigma _{n-1}\rangle$ and each tuple $(a_0,\ldots ,a_{n-1})$ such that $a_i\in j(\{\sigma _i\} )$ for each $i$, we have
$$G(\langle a_0,\ldots ,a_{n-1}\rangle )\;\in\; j(\{ F(\sigma )\} )$$\end{lemma}
\begin{proof} First we define $H$ such that for $a_0\in j(\{\sigma _0\} ),\ldots ,a_{n-1}\in j(\{\sigma _{n-1}\} )$ we have $H(\langle a_0,\ldots ,a_{n-1}\rangle )\in j(\{\sigma \} )$. Since $F:\{\sigma\}\to \{ F(\sigma )\}$ we have by monotony of $j$ an element of $\bigcap_{\sigma}[j(\{\sigma\} )\to j(\{ F(\sigma )\} )]$ so if we compose this with $H$ we have our desired function $G$.

Since $j$ is a local operator we have elements:
$$\begin{array}{rcl} c & \in & j(\{\langle\rangle\} ) \\
\beta & \in & \bigcap_{p,q}[j(p)\wedge j(q)\to j(p\wedge q)] \\
\gamma & \in & \bigcap_{\sigma ,a}{[}j(\{\sigma\}\wedge\{ a\} )\to j(\{\sigma\ast\langle a\rangle\} ){]} \end{array}$$
Define $G$ by recursion on $n$:
$$\begin{array}{rcl} G(\langle\rangle ) & = & c \\
G(\langle a_0,\ldots ,a_n\rangle ) & = & \gamma (\beta\langle G(\langle a_0,\ldots ,a_{n-1}\rangle ),a_n\rangle ) \end{array}$$
The trivial verification is left to the reader.\end{proof}
\begin{theorem}\label{Pittsarithmetic} Pitts' local operator, the local operator from~\ref{Pittsloc}, forces every arithmetical set of natural numbers to be decidable.\end{theorem}
\begin{proof} Let $\chi _D$ denote the characteristic function of a set $D$; to be specific let $\chi _D(n)=0$ if $n\in D$, and 1 otherwise. We write ${\uparrow} n$ for $\{ m\in\mathbb{N}\, |\, m\geq n\}$.

Let $g$ be the function which sends $p\subseteq\mathbb{N}$ to $\bigcup_n[({\uparrow} n)\to p]$, so Pitts' local operator is $L(g)$. Recall that $L(g)$ forces a set $D$ to be decidable if and only if there is a total recursive function which sends each $n$ to an element of $L(g)(\{\chi _D(n)\} )$. Let $\cal A$ be the class of sets forced by $L(g)$ to be decidable; then $\cal A$ contains the recursive sets and is closed under complements, so it suffices to see that $\cal A$ is closed under existential quantification: if $A\in {\cal A}$ then also $\exists A\in {\cal A}$, where
$$\exists A\; =\; \{ x\, |\, \exists n(\langle x,n\rangle\in A)\}$$
Let $F$ be the function which sends a sequence $\sigma =\langle\sigma _0,\ldots ,\sigma _{n-1}\rangle$ to 0 if at least for one $i$, $\sigma _i=0$, and to 1 otherwise. Let $G$ be the partial recursive function obtained by Lemma~\ref{Pittsarithmlemma}, with $L(g)$ for $j$.

Assuming $A\in {\cal A}$ let $F_A\in\bigcap_n[\{ n\}\to L(g)(\{\chi _A(n)\} )]$. For $x$ and $n$ consider the sequence
$$\langle F_A(\langle x,0\rangle ),\ldots ,F_A(\langle x,n\rangle )\rangle$$
We have $F_A(\langle x,i\rangle )\in L(g)(\{\chi _A(\langle x,i\rangle )\} )$. By using $G$ we construct a total recursive function $H$ such that for all $x,n$:
$$\begin{array}{rcll}
H(x)n & \in & L(g)(\{ 0\} ) & \text{if for some $m\leq n$, $\langle x,m\rangle\in A$} \\
H(x)n & \in & L(g)(\{ 1\} ) & \text{otherwise}\end{array}$$
We see that if for some $n$, $\langle x,n\rangle\in A$, then $H(x)k\in L(g)(\{ 0\} )$ for all sufficiently large $k$; if there is no $n$ with $\langle x,n\rangle\in A$ then $H(x)k\in L(g)(\{ 1\} )$ always. We conclude that
$$H(x)\; \in\; \bigcup_m[({\uparrow}m)\to L(g)(\{\chi _{\exists A}(x)\} )]$$
in other words, $H(x)\in g(L(g)(\{\chi _{\exists A}(x)\} ))$.

From the proof of \ref{Lchar} we know that there is an element 
$$b\in\bigcap_{p\subseteq\mathbb{N}}[g(L(g)(p))\to L(g)(p)]$$
Composing with $H(x)$ we get an element
$$\lambda x.b(H(x))\;\in\; \bigcap_x[\{ x\}\to L(g)(\{\chi _{\exists A}(x)\} )]$$
which was what we had to find.\end{proof}
\medskip

\noindent {\bf Open Problem}. Are the arithmetical sets {\em all\/} the sets which are forced to be decidable by Pitts' local operator?
\section{$\theta$-Realizability}
In this section we give a concrete presentation of a truth definition for first-order arithmetic in the subtopos of $\Eff$ determined by the local operator $L(G_{\theta})$, where $\theta :B\to {\cal P}^*{\cal P}(\mathbb{N})$. For background on the theory of triposes, the reader is referred to \cite{OostenJ:reaics}.

In general, if $R_X:P(X)\to P(X)$ is a local operator on a tripos $P$, the subtripos corresponding to $R$ can be presented as follows: the underlying set of the fibre over a set $X$ is just $P(X)$, and the order is given by the relation $\leq _R$ where $\phi\leq _R\psi$ if and only if $\phi\leq R(\psi )$ in the tripos $P$. Denoting this tripos by $(P,\leq _R)$, the inclusion into $(P,\leq )$ is given by the map $R$; its left adjoint is the identity function. This last map preserves $\wedge$, $\vee$ and $\exists$; if we denote implication and universal quantification in the subtripos by $\Rightarrow '$ and $\forall '$ respectively (and those in the original tripos by $\Rightarrow$, $\forall$), the relation is as follows:
$$\begin{array}{rcl} \phi\Rightarrow '\psi & \cong & \phi\Rightarrow R(\psi ) \\
\forall 'x\phi & \cong & \forall xR(\phi ) \end{array}$$
We can now give the truth definition in the form of a notion of realizability.

Recall from definition~\ref{jip} the notion `sight $S$ is {\em on\/} $\theta$'; from definition \ref{rdefdef} the notion `$r$-defined', and from \ref{sectordef} the notation $z[S]$.
\begin{definition}[$\theta$-realizability]\label{thetareal}\em Define a relation between numbers and sentences of arithmetic, pronounced `$n$ $\theta$-{\bf realizes} $\phi$', as follows, by induction on $\phi$:
\begin{itemize}
\item[] $n$ $\theta$-{\bf realizes} $t=s$ if and only if the equation $t=s$ is true;
\item[] $n$ $\theta$-{\bf realizes} $\phi\wedge\psi$ if and only if $n=\langle a,b\rangle$ and $a$ $\theta$-{\bf realizes} $\phi$ and $b$ $\theta$-{\bf realizes} $\psi$;
\item[] $n$ $\theta$-{\bf realizes} $\phi\vee\psi$ if and only if either $n=\langle 0,m\rangle$ and $m$ $\theta$-{\bf realizes} $\phi$, or $n=\langle 1,m\rangle$ and $m$ $\theta$-{\bf realizes} $\psi$;
\item[] $n$ $\theta$-{\bf realizes} $\phi\to\psi$ if and only if for every $m$ such that $m$ $\theta$-{\bf realizes} $\phi$, $nm$ is defined and there is a sight $S$ on $\theta$ such that $nm$ is $r$-defined on $S$ and for every $w\in (nm)[S]$, $w$ $\theta$-{\bf realizes} $\psi$;
\item[] $n$ $\theta$-{\bf realizes} $\neg\phi$ if and only if no number $\theta$-{\bf realizes} $\phi$;
\item[] $n$ $\theta$-{\bf realizes} $\exists x\phi (x)$ if and only if $n=\langle a,b\rangle$ and $b$ $\theta$-{\bf realizes} $\phi (a)$;
\item[] $n$ $\theta$-{\bf realizes} $\forall x\phi (x)$ if and only if for all $m$, $nm$ is defined and there is a sight $S$ on $\theta$ such that $nm$ is $r$-defined on $S$ and for every $w\in (nm)[S]$, $w$ $\theta$-{\bf realizes} $\phi (m)$.\end{itemize}\end{definition}
\begin{proposition}\label{thetarealprop} For $\theta$ as above, a sentence of first-order arithmetic is true in the subtopos of $\Eff$ determined by the local operator $L(G_{\theta})$, if and only if it has a $\theta$-realizer.\end{proposition}
\begin{theorem}\label{pitts=true} Let $j$ be a local operator in $\Eff$ such that $j\leq\neg\neg$ and $j$ forces every arithmetically definable subset of $\mathbb{N}$ to be decidable. Then the subtopos $\Eff _j$ of $\Eff$ determined by $j$ satisfies true arithmetic.\end{theorem}
\begin{proof} Truth of arithmetic in $\Eff _j$ is given by a realizability as in definition~\ref{thetareal}, which we call $j$-realizability in this proof. We shall not employ sights and simplify the clauses for $\to$ and $\forall$ to:\begin{itemize}\item[] $n$ $j$-{\bf realizes} $\phi\to\psi$ if and only if for every $m$ such that $m$ $j$-{\bf realizes} $\phi$, $nm$ is defined and $nm$ is an element of the set $j(\{ s\, |\, s\text{ $j$-{\bf realizes} $\psi$}\} )$
\item[] $n$ $j$-{\bf realizes} $\forall x\phi (x)$ if and only if for all $m$, $nm$ is defined and is an element of the set $j(\{ s\, |\, s\text{ $j$-{\bf realizes} $\phi (m)$}\} )$\end{itemize}
Since $j\leq\neg\neg$ we have $j(\emptyset )=\emptyset$ and therefore $n$ $j$-{\bf realizes} $\neg\phi$ if and only if no number $j$-{\bf realizes} $\phi$; and $n$ $j$-{\bf realizes} $\neg\neg\phi$ if and only if some number $j$-{\bf realizes} $\phi$. As a further simplification, we modify the definition so that for a string of universal quantifiers we have: $n$ $j$-{\bf realizes} $\forall x_1\cdots\forall x_n\phi$ if and only if for all $k_1,\ldots ,k_n$, $nk_1\cdots k_n$ (which we shall abbreviate as $n\vec{k}$) is defined and an element of $j(\{ s\, |\, s\text{ $j$-{\bf realizes} $\phi (k_1,\ldots ,k_n)$}\} )$.

Since $j$ is a local operator we can fix numbers $\alpha ,\beta ,\gamma ,\delta$ such that:
$$\begin{array}{rcl}\alpha & \in & \bigcap_{p,q\subseteq\mathbb{N}}(p\to q)\to (jp\to jq) \\
\beta & \in & \bigcap_{p\subseteq\mathbb{N}}p\to jp \\
\gamma & \in & \bigcap_{p\subseteq\mathbb{N}}jjp\to jp \\
\delta & \in & \bigcap_{p,q\subseteq\mathbb{N}}jp\wedge jq\to j(p\wedge q)\end{array}$$

We shall now prove by simulaneous induction on the structure of an arithmetical formula $\phi (x_1,\ldots ,x_n)$ the following statements:\begin{rlist}
\item\begin{alist}\item For all $k_1,\ldots ,k_n\in\mathbb{N}$: if there is a $j$-realizer for $\phi (k_1,\ldots ,k_n)$ then $\phi (k_1,\ldots ,k_n)$ is true in the standard model $\mathbb{N}$ in Set;
\item There is a partial recursive function $s_{\phi}$ of $n$ arguments, such that for all $k_1,\ldots ,k_n$: if $\phi (k_1,\ldots ,k_n)$ is true in $\mathbb{N}$ then $s_{\phi}(k_1,\ldots ,k_n)$ is defined and an element of $j(\{ s\, |\, s\text{ $j$-{\bf realizes} $\phi (k_1,\ldots ,k_n)$}\} )$;\end{alist}
\item There is a $j$-realizer for $\forall\vec{x}(\phi (\vec{x})\vee\neg\phi (\vec{x}))$.\end{rlist}

\noindent For atomic $\phi$, i)a) holds by definition of $j$-realizability; for i)b), let $s_{\phi}$ be $\lambda x_1\cdots x_k.\beta (0)$. The statement is obvious. Statement ii) is clear since in any topos, basic equations on the NNO are decidable.

\noindent Induction step i)a) for $\to$: suppose $m$ $j$-realizes $\phi (\vec{k})\to\psi (\vec{k})$ and $\phi (\vec{k})$ is true in $\mathbb{N}$. By induction hypothesis i)b) for $\phi$, $s_{\phi}(\vec{k})$ is defined and in $j(\{ s\, |\, s\text{ $j$-{\bf realizes} $\phi (\vec{k})$}\}$. Then
$$\alpha m(s_{\phi}(\vec{k}))\in jj(\{ s\, |\, s\text{ $j$-{\bf realizes} $\psi (\vec{k})$}\})$$
so since $j\emptyset =\emptyset$ we see that there exists a $j$-realizer for $\psi (\vec{k})$; hence by induction hypothesis i)a) for $\psi$, $\psi (\vec{k})$ is true.

\noindent Induction step i)b) for $\to$: define $s_{\phi\to\psi}$ by
$$s_{\phi\to\psi}(\vec{k})\, =\, \beta (\lambda m.s_{\psi}(\vec{k}))$$
The proof that this works is left to the reader.

\noindent Induction step ii) for $\to$ follows by logic from the induction hypotheses for $\phi$ and $\psi$.

\noindent Induction step i)a) for $\wedge$: follows readily from the induction hypotheses. For i)b), define
$$s_{\phi\wedge\psi}(\vec{k})\, =\, \delta (\langle s_{\phi}(\vec{k}),s_{\psi}(\vec{k})\rangle )$$
Again, induction step ii) follows by logic.

\noindent Induction step for $\vee$: i)a) follows easily from the induction hypotheses. For i)b), given $\phi (\vec{k})\vee\psi (\vec{k})$ let, by induction hypothesis ii) for $\phi$, $m$ be a $j$-realizer of $\forall\vec{x}(\phi (\vec{x})\vee\neg\phi (\vec{x}))$, so
$$m\vec{k}\,\in\, j(\{ s\, |\, s\text{ $j$-{\bf realizes} $\phi (\vec{k}\vee\neg\phi (\vec{k}))$}\}$$
Let $a$ be such that for all $\vec{k},y$:
$$a\vec{k}y\;\simeq\;\left\{\begin{array}{rl}y & \text{if }(y)_0=0 \\ \langle 1,s_{\psi}(\vec{k})\rangle & \text{if }(y)_0\neq 0\end{array}\right.$$
Define $s_{\phi\vee\psi}(\vec{k})\, =\, \alpha (a\vec{k})(m\vec{k})$. This satisfies the induction step: assume $\phi (\vec{k})\vee\psi (\vec{k})$ is true. Then whenever $y$ $j$-realizes $\phi (\vec{k})\vee\neg\phi (\vec{k})$, we have by induction hypothesis on $\phi$ and $\psi$, that $a\vec{k}y$ $j$-realizes $\phi (\vec{k})\vee\psi (\vec{k})$. Therefore $\alpha (a\vec{k})(m\vec{k})$ is an element of $j(\{ s\, |\, s\text{ $j$-{\bf realizes} $\phi (\vec{k})\vee\psi (\vec{k})$}\} )$, as desired.

\noindent Induction step ii) for $\vee$ again follows by logic.

\noindent Induction step for $\forall$: i)a) if $m$ $j$-realizes $\forall x\phi (\vec{k},x)$ then for all n, $mn$ is defined and an element of $j(\{ s\, |\, s\text{ $j$-{\bf realizes} $\phi (\vec{k},n)$}\}$; since $j\emptyset =\emptyset$, by the induction hypothesis for $\phi$ it follows that for all $n$, $\phi (\vec{k},n)$ is true; hence $\forall x\phi (\vec{k},x)$ is true.

\noindent For i)b) define $s_{\forall x\phi}(\vec{k})\, =\, \beta (\lambda y.s_{\phi}(\vec{k},y))$. Verification is easy.

\noindent For ii) let $A$ be the arithmetical set
$$\{\vec{k}\, |\, \text{for all $x\in\mathbb{N}$, $\phi (\vec{k},x)$ is true}\}$$
By assumption on $j$, $j$ forces this set to be decidable; let $a$ be such that for all $\vec{k}$, $a\vec{k}\in j(\{ 0\} )$ if $\vec{k}\in A$, and $a\vec{k}\in j(\{ 1\} )$ otherwise. Let $b$ be such that for all $\vec{k},v$:
$$b\vec{k}v\;\simeq\;\left\{\begin{array}{rl}\alpha (\lambda u.\langle 0,u\rangle )(s_{\forall x\phi}(\vec{k})) & \text{if }v=0 \\ \alpha (\lambda u.\langle 1,u\rangle )(\beta (0)) & \text{if }v\neq 0\end{array}\right.$$
Then if $v=0$ and $\vec{k}\in A$, it follows by step i)b) just proved, that 
$$b\vec{k}v\in j(\{\langle 0,s\rangle\, |\, s\text{ $j$-{\bf realizes} $\forall x\phi (\vec{k},x)$}\} )$$
and if $v=1$ and $\vec{k}\not\in A$ then by step i)a) just proved it follows that
$$b\vec{k}v\in j(\{\langle 1,s\rangle\, |\, s\text{ $j$-{\bf realizes} $\neg\forall x\phi (\vec{k},x)$}\} )$$
So when $v\in\{\chi _A(\vec{k})\}$ (where $\chi _A$ is the characteristic function of $A$) then
$$b\vec{k}v\in j(\{ s\, |\, s\text{ $j$-{\bf realizes} $\forall x \phi (\vec{k},x)\vee\neg\forall x\phi (\vec{k},x)$}\} )$$
Therefore, since $a\vec{k}\in j(\{\chi _A(\vec{k})\} )$ we have
$$\alpha (b\vec{k})(a\vec{k})\in jj(\{ s\, |\, s\text{ $j$-{\bf realizes} $\forall x\phi (\vec{k},x)\vee\neg\forall x\phi (\vec{k},x)$}\} )$$
so
$$\gamma(\alpha (b\vec{k})(a\vec{k}))\in j(\{ s\, |\, s\text{ $j$-{\bf realizes} $\forall x\phi (\vec{k},x)\vee\neg\forall  x\phi (\vec{k},x)$}\} )$$
and $\lambda\vec{k}.\gamma (\alpha (b\vec{k})(a\vec{k}))$ is thus a $j$-realizer for $\forall\vec{y}(\forall x\phi (\vec{y},x)\vee\neg\forall x\phi (\vec{y},x))$.

Induction step for $\exists$: i)a) follows at once from the induction hypothesis. We prove i)b) and ii) simultaneously. Clearly, from the induction hypotheses on $\phi$ it follows that $\exists x\phi (\vec{k},x)$ is true if and only it has a $j$-realizer. So the set $A=\{\vec{k}\, |\, \exists x\phi (\vec{k},x)\text{ has a $j$-realizer}\} =\{\vec{k}\, |\,\exists x\phi (\vec{k},x)\text{ is true}\}$ is arithmetical. By hypothesis on $j$, its characteristic function is forced to be total by $j$. Also, by induction hypothesis, the characteristic function of the set $\{\vec{k},v\, |\,\phi (\vec{k},v)\text{ has a $j$-realizer}\}$ is forced to be total by $j$. Since by Hyland's theorem (\ref{HylandTuring}) the set of functions which are forced to be total by $j$ is closed under `recursive in', the function
$$f(\vec{k})\; =\; \left\{\begin{array}{rl} 0 & \text{if for no $v$, $\phi (\vec{k},v)$ has a $j$-realizer} \\
m+1 & \text{if $m$ is least such that $\phi (\vec{k},m)$ has a $j$-realizer}\end{array}\right.$$
is forced to be total by $j$; let $a$ be such that for all $\vec{k}$, $a\vec{k}\in j(\{ f(\vec{k})\} )$.

If $\exists v\phi (\vec{k},v)$ is true hence $f(\vec{k})=m+1$ for some $m$, then by induction hypothesis i)b) on $\phi$, $\delta (\langle\beta (m),s_{\phi}(\vec{k},m)\rangle )$ is an element of $j(\{ s\, |\, s\text{ $j$-{\bf realizes} $\exists v\phi (\vec{k},v)$}\{ )$. It follows that
$$\alpha (\lambda n.\delta (\langle\beta (n-1),s_{\phi}(\vec{k},n-1)\rangle ))(a\vec{k})$$
is an element of $jj(\{ s\, |\, s\text{ $j$-{\bf realizes} $\exists v\phi (\vec{k},v)$}\} )$; so if we define $s_{\exists v\phi}(\vec{k})$ by
$$\gamma [\alpha (\lambda n.\delta (\langle\beta (n-1),s_{\phi}(\vec{k},n-1)\rangle ))(a\vec{k})]$$
then $s_{\exists v\phi}$ has the required property.

\noindent The proof that $\forall \vec{y}(\exists x\phi (\vec{y},x)\vee\neg\exists x\phi (\vec{y},x))$ has a $j$-realizer, is now straightforward (again, one uses the function $f$), and left to the reader.\end{proof}

\begin{small}
\bibliographystyle{plain}

\end{small}
\end{document}